\documentclass[a4paper,12pt]{amsart}
\usepackage[T1]{fontenc} 
\usepackage[all]{xy}
\usepackage[english]{babel}
\usepackage{frcursive} 
\usepackage{eucal}
\usepackage{graphicx}
\usepackage{epsfig}
\usepackage{mathrsfs}
\usepackage{amssymb}
\usepackage{amsxtra}
\usepackage{enumerate}
\input cyracc.def

\newtheorem{theorem}{Theorem}[section]

\newtheorem{proposition}[theorem]{ Proposition}
\newtheorem{lemma}[theorem]{ Lemma}

\newtheorem{definition}[theorem]{Definition}

\theoremstyle{remark}

\newtheorem*{example}{\it Example}

\def \1{\mathbb {1}}
\def \RM{\mathbb {R}}%        corps des reels
\def \NM{\mathbb{N}}%        entiers naturels
\def \ZM{\mathbb{Z}}%        entiers relatifs
\def \CM{\mathbb{C}}%        nombres complexes

\def \QM{\mathbb{Q}}%        nombres rationnels

 \def \Vol {{\rm Vol}}

\def \Spec {{\rm Spec\,}}

\def \Aut {{\rm Aut\,}}

\def \p {{\rm exp\,}}

\def \d{\partial}%derivee partielle
 
\def\a{\alpha}

\def\l{\lambda}

\def\p{\varphi}

\def \s{\sigma}

\def \to{\longrightarrow} 
\def \w{\wedge}

\def \< {{\thengle }}
\def \> {{\rangle }}
\def \( {\left( }
\def \) {\right) }

\newcommand{\Ct}{{\mathcal C}}

\newcommand{\Ft}{{\mathcal F}}

\newcommand{\Lt}{{\mathcal L}}

\renewcommand{\mod}{{\rm  mod\,}}
\title[THE HERMAN CONJECTURE]{THE HERMAN CONJECTURE}
\author{  Mauricio  Garay}
\address{ Institut für Mathematik\\
FB 08 - Physik, Mathematik und Informatik\\
Johannes Gutenberg-Universität Mainz\\
Staudinger Weg 9\\
55128 Mainz.}
\begin{document}
\begin{abstract}  We show that KAM versal deformation theory answers positively Herman's invariant tori conjecture.
\end{abstract}
\maketitle
%\noindent {\em  Les séries ne pourraient-elthe pas, by exemple, converger quand $x^0_1$ and $x^0_2$ ont été choisies of telle sorte que le rapport
%$n_1/n_2$ soit incommensurable, and que son carré soit au contraire commensurable (\dots)?}
%\begin{flushright} H. Poincaré, {\em Méthodes mathématiques of the mécanique céleste.} \end{flushright}
%\bigskip
%\begin{flushright}{\em \`A la mémoire de Michel Herman.}\end{flushright}
%\bigskip
%\tableofcontents
\section{Introduction}
Two types of motions are usually observed in Hamiltonian systems: chaotic motions satisfying the ergodic hypothesis and quasi-periodic ones which stay on a confined region. The KAM theorem provides a criterion for a system to be  of the second type on a positive measure subset~\cite{Arnold_KAM,Kolmogorov_KAM,Moser_KAM}.

 In the nineties, Herman asked for the existence of invariant tori in more general situations than that of the KAM theorem. In 1998, during his ICM lecture, among other things, he made the following remarkable conjecture  for discrete time hamiltonian systems~\cite{Herman_ICM}:\\
  
{\em    In the neighbourhood  of a diophantine elliptic fixed point, a real analytic symplectomorphism has a positive measure set of invariant tori.}\\

According to his students, Herman  also  asked whether the following two variants for continuous time hamiltonian dynamics hold:\\

{\em    In the neighbourhood of a torus carrying a quasi-periodic motion with diophantine frequency,  a real analytic hamiltonian has a positive measure set of invariant tori.}\\

{\em   In the  neighbourhood of an elliptic critical point with diophantine frequency,  a real analytic  hamiltonian has a positive measure set of invariant tori.}\\  

If these conjectures turned out to be true then they will give strong restrictions for ergodicity. The purpose of this paper is to show that these are direct consequences of  KAM versal deformation theory which I developped in~\cite{Abstract_KAM,arithmetic,Lagrange_KAM}. Therefore, we assume the reader to be acquainted with these papers. 

 To state our theorem, let us first recall the definition of density, in the measure theoretical sense. 
 
  For $\a \in \RM^n$, we denote by $B(\a,r)$ the ball centred at $\a$ with radius~$r$. The {\em density} of a measurable subset $K \subset \RM^n$ at a point $\a$ is the limit~(if it exists):
$$\lim_{r \to 0} \frac{\Vol(K \cap B(\a,r))}{\Vol(B(\a,r))}.$$

 Recall that a vector $\a \in \RM^n$ satisfies {\em Bruno arithmetical} condition if
  $$\sum_{k \geq 0}{\frac{ \log a_k}{2^k}}>-\infty  $$
 where
 $$a_k=\min \{ (j,\a) : \| j \| \leq 2^k,\ j \in \ZM^n \} $$
 and $(-,-)$ denotes the euclidean scalar product~\cite{Brjuno}.
  We will prove the
 \begin{theorem}
 \label{T::Herman}
 Let $H:(\RM^{2n},0) \to (\RM,0)$ be an analytic function germ having an elliptic critical point at the origin with frequency $\a$. If $\a$ satisfies Bruno arithmetical condition then for any representative of the germ $H$, the invariant tori form a set of density one at the origin (and in fact at any of its point near the origin).
\end{theorem}
We will concentrate on the formulation for critical points, the cases of periodic orbits or neighbourhoods of KAM tori are similar. Moreover, it will be clear that the real structure has no real influence on the problem: the KAM tori are just the real part of complex lagrangian invariant manifolds. 
 
%%%%%%%%%%%%%%%%%%%%%%%%%%%
%%%%%%%%%%%%%%%%%%%%%%%%%%%%%%%%%%%%
\section{Formal normal forms}
%%%%%%%%%%%%%%%
The relation between KAM versal deformation theory and the Hermann conjecture can already be seen at the level of formal power series.
Therefore, we first recall basic perturbation theory which goes back to the ${\rm XIX^{th}}$ century and then extend it to the relative case, following
Grothendieck's philosophy which considers that a family of varieties is a scheme relative to a given base.

So, we consider the graduation by the polynomial degree in the algebra of formal power series
 $$\CM[[q,p]]:=\CM[[q_1,\dots,q_n,p_1,\dots,p_n]].$$ 
 Note that the graduation of an algebra always extends to the module of its derivations: {\em a derivation is homogeneous of degree $k$ if it maps $Gr^i (A)$ to $Gr^{i+k}(A)$.} For instance, the derivations $\d_{q_1}$, $\d_{p_i}$ are homegeneous of degree $-1$, meaning simply that the derivative decreases the degree of a polynomial by one.

We write $f=g+o(l) $ if $f-g$ contains only terms of degree higher that $l$.  
\begin{proposition} 
\label{P::Delaunay}
Let $H \in \CM[[q,p]]  $ be of the form
$$H=\sum_{i=1}^n \a_i p_iq_i+o(2) $$
 If the $\a_i$'s are linearly $\QM$-independent then, for any $k>0$, there exists a symplectomorphism
 $\p_k:(\CM^{2n},0) \to (\CM^{2n},0) $
 and a polynomial $P_k \in \CM[X_1,\dots,X_n] $
 such that
 $$H \circ \p_k(q,p)=P_k(q_1p_1,\dots,q_np_n)+o(k)$$
 Moreover, the polynomial $P_k$ does not depend on the choice of $\p_k$. 
 \end{proposition}
 \begin{proof}
Write
  $$H=H_0+R_3+o(3),\ H_0:=\sum_{i=1}^n \a_i p_iq_i.$$
  As  the $\a_i$ are linearly $\QM$ independent, there exists a homogeneous polynomial $f_3$ of degree $3$ such that:
  $$\{ H_0, f_3\} =R_3.$$ 
  The symplectic automorphism
  $$e^{\{ -, f_3\}} :\CM[[q,p]] \to \CM[[q,p]],\ F \mapsto F+\{ F, f_3\}+\frac{1}{2!}\{ \{ F, f_3\}, f_3\}+\dots $$
 maps $H$ to
  $$ e^{\{ H, f_3\}}=\sum_{i=1}^n \a_i p_iq_i+o(3).$$
  This proves the proposition for $k=3$.
  
 Assume the proposition is proved up to some odd number $k \geq 3$. We may assume that
 $$H=P_k(q_1p_1,\dots,q_np_n)+R_{k+1}+R_{k+2}+o(k+2)$$
 where $R_j$ is a homogeneous polynomial of degree $j$. There exist  homogeneous polynomials $f_{k+1},f_{k+2}$ of respective degrees $k+1, k+2$ and a homogeneous polynomial $B_k \in \CM[X_1,\dots,X_n]$ of degre $(k+1)/2$  such that:
  $$\{ H, f_{k+1}\} =R_{k+1}+B_k(p_1q_1,\dots,p_nq_n), \ \{ H, f_{k+2}\} =R_{k+2} .$$ 
We have:
$$e^{\{ H, f_{k+1}+f_{k+2}\}}=H+(P_{k}+B_k)(p_1q_1,\dots,p_nq_n)+o(k+2).$$This proves the proposition.
 \end{proof}
 This proposition is of course standard. Poincaré  attributed this result  to Delaunay and Lindstedt, but it is sometimes called the {\em Birkhoff normal form}~\cite{Birkhoff,Poincare_Methodes}.  
 
 We now consider a variant with parameters of the above result.
Consider the algebra of formal power series  $A=\CM[[\l,q,p]]$ with $q=(q_1,\dots,q_n)$, $p=(p_1,\dots,p_n)$,  $\l=(\l_1,\dots,\l_n)$.  We  extend the graduation to $A$ by putting the degree of the $\l_i$'s equal to $2$.   
The bivector $v=\sum_{i=1}^n \d_{q_i} \w \d_{p_i} $ 
induces a $\CM[[\l]]$-linear Poisson structure on $A:=\CM[[\l,q,p]]$:
$$\{ f,g \}:=\sum_{i=1}^n \d_{q_i}f \d_{p_i}g-\d_{p_i}f\d_{q_i} g. $$
 
 Denote by $I \subset \CM[[\l,q,p]]$ the ideal generated by the $p_iq_i-\l_i$'s for $i=1,\dots,n$. If two series are equal modulo the square of the ideal $I$ then they define the same hamiltonian derivation of $\CM[[\l,q,p]]/I$.  
 
The inclusion $\CM[[q,p]] \subset \CM[[\l,q,p]]$ induces an isomorphism of Poisson algebras
$$\CM[[q,p]] \overset{\sim}{\to}  \CM[[\l,q,p]]/I .$$
In particular, any hamiltonian derivation in $\CM[[q,p]]$ can be identified with a hamiltonian derivation in $\CM[[\l,q,p]]/I$. We get the following parametric variant of Proposition \ref{P::Delaunay}:

\begin{proposition}
\label{P::Lagrange}
 Consider an element $H \in \CM[[\l,q,p]] $ 
such that 
$$H(\l,q,p)=\sum_{i=1}^n \a_i p_iq_i+o(2),\ \a_i \in \CM $$
 If the $\a_i$'s are linearly independent over $\QM$ then there exists a Poisson automorphism
 $\p \in \Aut(\CM[[\l,q,p]])$
 and  $\widehat{g_1},\dots,\widehat{g_n} \in \CM[[\l]]$  such that
 $$\p(H)=\sum_{i=1}^n (\a_i+\widehat{g_i}(\l)) p_iq_i\ (\mod I^2 \oplus \CM[[\l]])$$
  where $I \subset \CM[[\l,q,p]]$ is the ideal generated by $p_1q_1-\l_1,\dots,p_nq_n-\l_n$
 \end{proposition}
The relation with Proposition~\ref{P::Delaunay} is straightforward. Take 
$$H=\sum_{i=1}^n \a_i p_iq_i+o(2) \in \CM[[q,p]]$$ 

Proposition \ref{P::Delaunay} implies the existence of a symplectic automorphism $\p$ and a power series $P \in \CM[[X_1,\dots,X_n]]$ such that
$$\p(H)=P(q_1p_1,\dots,q_np_n). $$
Now, embed $\CM[[q,p]]$ in $\CM[[\l,q,p]]$. The Taylor formula at order one gives:
$$P(q_1p_1,\dots,q_np_n)=P(\l_1,\dots,\l_n)+\sum_{i=1}^n  g_i(\l) (q_ip_i-\l_i) \mod I^2   $$
where $g_i=\d_{\l_i} P$.  In particular, as derivation of $\CM[[\l,q,p]]/I$, we have
$$\{ -,P(q_1p_1,\dots,q_np_n) \}=\{ -, \sum_{i=1}^n  g_i(\l) (q_ip_i-\l_i) \}. $$ 
This shows that the parametric normal form is equivalent to the usual one for functions which do not depend on parameters.

The above proposition shows  that the parameter $\l$ is in fact two-fold: it parametrises the frequency of the Hamiltonian and the Lagrangian invariant variety at the same time. In order to distinguish both roles, we introduce a new parameter $t=(t_1,\dots,t_n)$ to parametrise the frequency.

So, we consider now the algebra of formal power series in $4n$ variables $ \CM[[t,\l,q,p]]$ with $t=(t_1,\dots,t_n)$. We assign the degree $0$ to the parameters $t_1,\dots,t_n$ which parametrise the frequency of the Hamiltonian. So that the polynomials
$$\sum_{i=1}^n (\a_i+t_i) p_iq_i,\ p_iq_i-\l_i,\ i=1,\dots,n$$
are both homogeneous of degree $2$.

 Proposition~\ref{P::Delaunay} admits the following variant:
 \begin{proposition}
\label{P::parametric} For any $H \in \CM[[t,\l,q,p]] $ 
such that 
$$H(t,\l,q,p)=\sum_{i=1}^n (\a_i+t_i) p_iq_i+o(2),\ \a_i \in \CM $$
 If the $\a_i$'s are linearly independent over $\QM$ then there exists a Poisson automorphism
 $\p \in \Aut(\CM[[t,\l,q,p]])$ such that
 $$\p(H)=\sum_{i=1}^n (\a_i+t_i) p_iq_i (\mod I^2 \oplus \CM[[t,\l]])$$
 where $I \subset \CM[[t,\l,q,p]]$ is the ideal generated by $p_1q_1-\l_1,\dots,p_nq_n-\l_n$.
 \end{proposition}
 
 \section{The frequency space}

  In the situation of Proposition~\ref{P::Lagrange}, there exists a unique vector space $\Ft(H)$ generated by the partial derivative of $\widehat{g}$ which contains the image of $\widehat{g}$. This provides a formal variant of non-degeneracy conditions and we will see that formal non-degeneracy implies $C^\infty$-non degeneracy.  
   
 From a scheme-theoretical point of view,we have a map of formal schemes
$$\widehat{g}:\Spec(\CM[[\l]]) \to \Spec(\CM[[t]]) $$
induced by substituting $t_i$ with $\widehat{g_i}(\l)$ .
\begin{definition} The smallest vector space which contains the image of the formal frequency map $\widehat{g} $ is called the frequency space of $H$.
We denote it by  $\Ft(H)$.
\end{definition}
\begin{example} Consider the function 
$$ H(q,p)=\a_1\,p_1q_1+\a_2 p_2q_2+p_1^2q_1^2,\ \a_1,\a_2 \in \CM.$$
We have:
$$H(q,p)=(\a_1+2\l_1)p_1q_1+\a_2 p_2q_2 \mod (I^2 \oplus \CM[\l])$$
where $I$ is the ideal generated by $p_1q_1-\l_1$ and $p_2q_2-\l_2$. The frequency map is defined by
$$\widehat{g_1}(\l)= 2\l_1,\ \widehat{g_2}(\l)=0. $$
Thus $\Ft(H)$ can be canonically identified with the first coordinate axis:
$$\{ (\l_1,\l_2):\l_2=0 \} \subset \CM^2.$$
\end{example}
 Note that the frequency space is invariant under the action of Poisson automorphisms. We sometimes abuse notations and also denote by $\Ft(H)$
 its embedding in $\CM^{3n}$ given by the inclusion $\CM[[\l]] \subset \CM[[\l,q,p]] $.

Although elementary, the following lemma turns out to be essential:
 \begin{lemma}
 \label{L::fondamental} Let  $e_1,\dots,e_n$ a basis of $\CM^n$, fix $k \leq n$ and denote by $V$ the vector space generated by
 $e_1,\dots,e_k \in \CM^n$. Let
 $$g:\Spec(\CM[[\l]]) \to \Spec(\CM[[t]]) $$
 be a formal mapping whose image lies in $V$.
 Put
$$G= \sum_{i=1}^n (\a_i+t_i) (f_i,e_i)+R+o(2^{n+1})  \in \CM[[t,\l,q,p]],\  \deg(R)<2^{n+1}$$
and assume that the frequency space of $G$ restricted to $t=g$ is contained in $V$.
 Any hamiltonian vector field $v $ of degree $2^{n-1}$ such that
 $$v( G)=R+o(2^{n+1})  (\mod I^2)$$
 is tangent to $V $.
 \end{lemma} 
 \begin{proof}
 In the expansion
 $$e^{-v}G=G-v(G)+\frac{1}{2!} v(v(G))+\dots $$
 the terms $v^k(G)=v^{k-1}(R)+o(2^{n+1})$ are of order at least $2^{n+1}$ if $k>1$.
 The frequency space of $G(t=g(\l),-)$ is generated by $e_1,\dots,e_k$ thus  $R$ must be of the form
 $$R=\sum_{i=1}^k a_i(\l,t) p_iq_i+\sum_{i>k,j>k} t_j b_{i,j}(\l,t) p_iq_i +\sum_{I \neq J} b_{I,J}p^Iq^J+o(2^{n+1})\ (\mod I^2),$$
 and as $v$ is of degree $2^{n-1}$, we have
 $$v(t_i)=\left\{ \begin{matrix} a_i(\l,t)&{\rm for}& i \leq k  \\ \sum_{j>k} t_j b_{i,j}(\l,t) &{\rm for}& i > k\end{matrix} \right.$$
 This proves the Lemma. 
 \end{proof}

%%%%%%%%%%%%%%%%%%%%%%%%%%%%%%%%%%%%%%
 \section{KAM versal deformations}
 We now give an exact variant of Proposition~\ref{P::Lagrange}. We refer to \cite{Abstract_KAM,Lagrange_KAM} for the formalism of Arnold spaces that we shall now use.
 
 Let $a$ be a decreasing sequence and $\a \in \CM^n$ a vector which belongs to the arithmetic class $ \Ct(a)$.
 Given a map
$$g:(\CM^k,0) \to (\CM^n,0) $$
there is a natural pull-back notion for the Arnold space $C^l_\a(a)$, $l \in \NM$, these are the $C^l$-function defined on
$g^{-1}(\Ct(a)_n)$:
$$  f \in g^{-1}C^l_\a(a) \iff f \circ g \in C^l_\a(a).$$

\begin{theorem} 
\label{T::versal}
Let $a=(a_n)$ be a decreasing positive sequence such that
$$\sum_{n\geq 0} \frac{\log a_n}{2^n}>-\infty. $$ 
Let  $I \subset A$ be the   involutive ideal generated by the $p_iq_i+\l_i$'s and 
$$ H_0=\sum_{i=1}^n \a_i p_iq_i$$
with $(\a_1,\dots,\a_n) \in \Ct(a)$. 
For any $H$ such that $H=H_0+o(3)$, there exists a $C^\infty$-map germ
$$g=(g_1,\dots,g_n):(\CM^k,0) \to (\Ft(H),0) $$ and a Poisson morphism of  
$$\p:  \CM\{\l, q,p \}  \to  \CM\{q,p \} \hat \otimes  g^{-1}C^\infty_\a(a)    $$
such that
$$\p(H)= \sum_{i=1}^n (\a_i+g_i(\l)) p_iq_i (\mod I^2\oplus  g^{-1}C^\infty_\a(a)).$$
Moreover, if $H$ is real analytic for some complex anti-holomorphic involution then $\p$ and $g$ can also be chosen real.
\end{theorem}
\begin{proof}
We repeat Martinet's trick in our context~\cite{Martinet}.
We define
$$G=H+\sum_{i=1}^n t_i\, p_iq_i \in \CM\{ t,\l,q,p\} $$
and apply the
\begin{theorem}[\cite{Lagrange_KAM}]
\label{T::stable}
Let $a=(a_n)$ be a decreasing positive sequence such that
$$\sum_{n\geq 0} \frac{\log a_n}{2^n}>-\infty. $$  
Take $\a \in \Ct(a) \subset \CM^n$ and consider the morphism of Poisson algebras induced by the inclusion
$\CM\{ \l \} \subset C_{\a}^\infty(a)$:
$$ r:\CM\{t,\l,q,p\} \to \CM\{\l, q,p \} \hat \otimes C_{\a}^\infty(a)  .$$ 
Let  $I \subset \CM\{t,\l,q,p\} $ be the   involutive ideal generated by the
$p_iq_i+\l_i$'s and consider a holomorphic function of the type
$$G=\sum_{i=1}^n (\a_i+t_i) p_iq_i+o(2) \in \CM\{t,\l,q,p\}.$$
There exists a sequence a $1$-bounded morphism $u_\bullet$ such that
\begin{enumerate}[{\rm i)}]
\item  $u_k$ is a polynomial derivation of order $2^{n-1}$ and degree at most $2^n$~;
\item the sequence $\p_k=(e^{u_k}\dots e^{u_0}) $ converges in $\Lt( \CM\{t,\l,q,p\},\CM\{\l, q,p \} \hat \otimes C_{\a}^\infty(a) )$ to a
Poisson morphism ;
\item $\p(G)=\sum_{i=1}^n (\a_i+t_i) p_iq_i\ (\mod(r(I^2) \oplus   \CM\{\l \} \hat \otimes C_{\a}^\infty(a)).$
\end{enumerate}
Moreover if $G$ is real for some antiholomorphic involution then the $u$ can also be chosen real.
\end{theorem}
 We apply the theorem to our situation, we get a Poisson morphism
 $$\p=\lim e^{u_k}\dots e^{u_0} $$
 such that
 $$\p(G)=H_0+\sum_{i=1}^n t_i\, p_iq_i \ \mod(r(I^2) \oplus   \CM\{\l \} \hat \otimes C_{\a}^\infty(a)). $$
 We denote by $J_\bullet \subset E_\bullet$ the ideal of function which vanish on the frequency space.

 \begin{lemma}
 \label{L::restriction} The morphism $\p$ sends $J_0$ to $J_\infty$.
  \end{lemma}
 \begin{proof}
 Define 
 $$\p_k:= e^{u_k}\dots e^{u_0}.$$
By the implicit function theorem the ideal generated by the $\p_k(t_i)$ admits a system of generators of the form (it is possible to apply the theorem
uniformly in $n$ because the sequence is convergent):
$$t_1-g_1^k(\l,),\dots,t_n-g_n^k(\l). $$
 Lemma~\ref{L::fondamental} applied to $\p_k(G)$ shows by induction on $k$ that $u_k$ is tangent to the frequency space of $H$ that is
 $$u_k(J_k) \subset J_k. $$
 This proves the lemma.
 \end{proof}
 The previous lemma shows  that the restriction to $\Ft(H)$ commutes with $\p$:
$$\p(f_{\mid \Ft(H)})=\p(f)_{\mid \Ft(H)} $$
 By the Whitney extension theorem~\cite{Whitney_extension}, the morphism $\p$ is obtained by composition of a  morphism:
$$\Phi:E_0 \to E_0 $$
with the restriction morphism 
$$r:E_0 \to E_\infty.$$ 
As $\Phi(t_i)=t_i+o(2)$ for $i \leq k$, the implicit function theorem implies that the ideal generated by the $\Phi(t_i),\ i=1,\dots,k$ admits a system of generators of the form $$t_1-g_1(\l,),\dots,t_k-g_k(\l)$$
and for any $f \in E_0$, we have
$$\p(f_{\mid t=0})=\p(f)_{\mid t=g(\l)},$$
with $g=(g_1,\dots,g_k,0,\dots,0)$. In particular
$$ \p(H) = H_0+\sum_{i=1}^k g_i(\l) (f_i,e_i) \ \mod(r(I^2) \oplus   g^{-1}C^\infty_{\a}(a)).$$
This proves the theorem.
\end{proof}
Now fix $\tau$ and consider the sequence
$$a_n:=\s(\a_n)2^{-\tau n}.$$
By Theorem~\ref{T::versal}, the real part of the set $g^{-1}(\Ct(a))$ parametrises invariant lagrangian tori (and their degenerations)
 and, for $\tau$ big enough, by the Arithmetic Density Theorem, this set has density one at the origin~\cite{arithmetic}. This proves Theorem~\ref{T::Herman} and answers positively the Herman conjecture. 

 \bibliographystyle{amsplain}
\bibliography{master}

\providecommand{\bysame}{\leavevmode\hbox to3em{\hrulefill}\thinspace}
\providecommand{\MR}{\relax\ifhmode\unskip\space\fi MR }
\providecommand{\MRhref}[2]{%
  \href{http://www.ams.org/mathscinet-getitem?mr=#1}{#2}
}
\providecommand{\href}[2]{#2}
\begin{thebibliography}{10}

\bibitem{Arnold_KAM}
V.I. Arnold, \emph{{ Proof of a theorem of A. N. Kolmogorov on the preservation
  of conditionally periodic motions under a small perturbation of the
  Hamiltonian}}, Uspehi Mat. Nauk \textbf{18} (1963), no.~5, 13--40, English
  translation: Russian Math. Surveys.

\bibitem{Birkhoff}
G.~D. Birkhoff, \emph{Dynamical systems}, American Mathematical Society
  Colloquium Publications, vol.~IX, American Mathematical Society, Providence,
  R.I., 1927.

\bibitem{Brjuno}
A.D. Brjuno, \emph{{Analytic form of differential equations I}}, Trans. Moscow
  Math. Soc. \textbf{25} (1971), 131--288.

\bibitem{Abstract_KAM}
M.D. Garay, \emph{{An abstract KAM theorem}}, ArXiv 0910.3500.math.DS, To
  appear in Moscow Mathematical Journal.

\bibitem{arithmetic}
\bysame, \emph{{Arithmetic density}}, ArXiv: 1204.2493, 2012.

\bibitem{Lagrange_KAM}
\bysame, \emph{{Degenerations of invariant Lagrangian manifolds}}, ArXiv
  e-prints:1309.4028 (2013).

\bibitem{Herman_ICM}
M.R. Herman, \emph{Some open problems in dynamical systems}, Proceedings of the
  International Congress of Mathematicians, Doc. Math., vol.~II, 1998,
  pp.~797--808.

\bibitem{Kolmogorov_KAM}
A.~N. Kolmogorov, \emph{On the conservation of quasi-periodic motions for a
  small perturbation of the hamiltonian function}, Dokl. Akad. Nauk SSSR
  \textbf{98} (1954), 527--530.

\bibitem{Martinet}
J.~Martinet, \emph{Singularities of smooth functions and maps}, Lecture Notes
  Series, vol.~58, Cambridge University Press,, 1982, 272 pp.

\bibitem{Moser_KAM}
J.~Moser, \emph{{On the construction of almost periodic solutions for ordinary
  differential equations (Tokyo, 1969)}}, Proc. Internat. Conf. on Functional
  Analysis and Related Topics, Univ. of Tokyo Press, 1969, pp.~60--67.

\bibitem{Poincare_Methodes}
H.~Poincar\'e, \emph{Les m\'ethodes nouvelles de la m\'ecanique c\'eleste},
  Gauthiers-Villars, 1892-1899, 3 vol.

\bibitem{Whitney_extension}
H.~Whitney, \emph{Analytic extensions of differentiable functions defined in
  closed sets}, Trans. Amer. math. Soc. \textbf{36} (1934), 63--89.

\end{thebibliography}
 \end{document}